%% file: paper.tex
\documentclass[11pt]{article}
\usepackage{fullpage}
\usepackage[round]{natbib}

\usepackage{amsmath,amsthm,amsfonts,amssymb}
\newtheorem{theorem}{Theorem}
\newtheorem{lemma}{Lemma}
\input{paper-macros}

\usepackage{xcolor}

\title{On the sample complexity of parameter estimation in logistic regression with normal design}

\author{%
  Daniel Hsu \\ {\sl Columbia University} \and%
  Arya Mazumdar \\ {\sl University of California, San Diego}%
}

\begin{document}

\maketitle

\begin{abstract}
  \input{paper-abstract}

\end{abstract}

\input{paper-main}

\bibliographystyle{plainnat}
\bibliography{refs}

\newpage
\appendix

\input{paper-appendix}

\end{document}

%% file: paper-macros.tex
\usepackage{dsfont}
\usepackage{hyperref,cleveref}

\Crefformat{equation}{#2(#1)#3}
\crefformat{equation}{#2(#1)#3}
\Crefrangeformat{equation}{#3(#1)#4--#5(#2)#6}
\crefrangeformat{equation}{#3(#1)#4--#5(#2)#6}
\Crefmultiformat{equation}{#2(#1)#3}{ and~#2(#1)#3}{, #2(#1)#3}{ and~#2(#1)#3}
\crefmultiformat{equation}{#2(#1)#3}{ and~#2(#1)#3}{, #2(#1)#3}{ and~#2(#1)#3}

\usepackage{mathtools}
\mathtoolsset{centercolon}
\DeclarePairedDelimiterXPP\ind[1]{\mathds{1}}{\lbrace}{\rbrace}{}{#1}
\DeclarePairedDelimiterX\eval[1]{\lbrace}{\rvert}{#1 \delimsize\rbrace}

\DeclarePairedDelimiter\abs{\lvert}{\rvert}
\DeclarePairedDelimiter\card{\lvert}{\rvert}
\DeclarePairedDelimiter\norm{\lVert}{\rVert}
\DeclarePairedDelimiter\ceil{\lceil}{\rceil}
\DeclarePairedDelimiter\floor{\lfloor}{\rfloor}
\DeclarePairedDelimiter\del{\lparen}{\rparen}
\DeclarePairedDelimiter\sbr{\lbrack}{\rbrack}
\DeclarePairedDelimiter\cbr{\lbrace}{\rbrace}
\DeclarePairedDelimiter\set{\lbrace}{\rbrace}

\DeclarePairedDelimiter\intco{\lbrack}{\rparen}
\DeclarePairedDelimiter\intoo{\lparen}{\rparen}

\DeclareMathOperator\E{\mathbb{E}}

\DeclareMathOperator\Normal{N}
\DeclareMathOperator\sign{sign}
\DeclareMathOperator\Bernoulli{Bern}
\DeclareMathOperator\KL{KL}
\DeclareMathOperator\err{err}
\DeclareMathOperator\emperr{\widehat{err}}
\DeclareMathOperator\dif{d\!}
\DeclareMathOperator*\argmin{\arg\min}
\DeclareMathOperator*\argmax{\arg\max}

\newcommand\T{{\scriptscriptstyle{\mathsf{T}}}}
\newcommand\R{\mathbb{R}}

\def\ddefloop#1{\ifx\ddefloop#1\else\ddef{#1}\expandafter\ddefloop\fi}
\def\ddef#1{\expandafter\def\csname bf#1\endcsname{\ensuremath{\mathbf{#1}}}}
\ddefloop abcdefghijklmnopqrstuvwxyzABCDEFGHIJKLMNOPQRSTUVWXYZ\ddefloop
\def\ddef#1{\expandafter\def\csname bf#1\endcsname{\ensuremath{\boldsymbol{\csname #1\endcsname}}}}
\ddefloop {alpha}{beta}{gamma}{delta}{epsilon}{varepsilon}{zeta}{eta}{theta}{vartheta}{iota}{kappa}{lambda}{mu}{nu}{xi}{pi}{varpi}{rho}{varrho}{sigma}{varsigma}{tau}{upsilon}{phi}{varphi}{chi}{psi}{omega}{Gamma}{Delta}{Theta}{Lambda}{Xi}{Pi}{Sigma}{varSigma}{Upsilon}{Phi}{Psi}{Omega}{ell}\ddefloop
\def\ddef#1{\expandafter\def\csname cal#1\endcsname{\ensuremath{\mathcal{#1}}}}
\ddefloop ABCDEFGHIJKLMNOPQRSTUVWXYZ\ddefloop

\newcommand\fano{\delta}
\newcommand\klbound{\kappa}
\newcommand\Params{\Theta}

\newcommand\param{\theta}
\newcommand\bfparam{\bftheta}
\newcommand\truth{\param^\star}
\newcommand\erm{{\hat{\param}_{\operatorname{ERM}}}}
\newcommand\linest{{\hat{\param}_{\operatorname{linear}}}}
\newcommand\reluest{{\hat{\param}_{\operatorname{relu}}}}

\newcommand\estimate{{\hat{\param}}}
\newcommand\invtemp{\beta}
\DeclareMathOperator{\opt}{opt}
\newcommand\sphere{\ensuremath{S^{d-1}}}
\newcommand\prior{\Pi}

%% file: paper-abstract.tex
The logistic regression model is one of the most popular data generation model in noisy binary classification problems.
In this work, we study the sample complexity of estimating the parameters of the logistic regression model up to a given $\ell_2$ error, in terms of the dimension and the inverse temperature, with standard normal covariates.
The inverse temperature controls the signal-to-noise ratio of the data generation process.
While both generalization bounds and asymptotic performance of the maximum-likelihood estimator for logistic regression are well-studied, the non-asymptotic sample complexity that shows the dependence on error and the inverse temperature for parameter estimation is absent from previous analyses.
We show that the sample complexity curve has two change-points in terms of the inverse temperature, clearly separating the low, moderate, and high temperature regimes.

%% file: paper-main.tex
\section{Introduction}

This paper studies the sample complexity of estimating the parameter vector in the logistic regression model under a normal design, with particular attention paid to the dependence on the \emph{dimension} $d$, the \emph{inverse temperature} $\invtemp\geq0$, and the \emph{target error} $\epsilon \in \intoo{0,1}$.
We show how the form of the sample complexity changes depending on the particular relationship between the inverse temperature and the target error.

Our statistical model is as follows.
The parameter space is the unit sphere $\sphere = \set{ \param \in \R^d : \norm{\param} = 1 }$, where $\norm{\cdot}$ denotes the $\ell_2$ norm on $\R^d$.
The covariate vector $\bfx$ is a $d$-dimensional standard normal random vector.
Conditional on $\bfx$, the response $\bfy$ is a binary $\{-1,1\}$-valued (Bernoulli) random variable.
If the parameter vector in our model is $\param \in \sphere$, then  $\bfy =1$ with probability $g'(\invtemp\bfx^\T\param)$, where $g'$ is the standard logistic function $g'(\eta) = 1/(1+e^{-\eta})$, which is the derivative of the log partition function $g(\eta) = \ln(1+e^\eta)$.
The inverse temperature $\invtemp$, which is the norm of the coefficient vector on $\bfx$ appearing in the mean parameter, is regarded as a model hyperparameter, and it governs the signal-to-noise ratio of this data generation process. In particular, when $\invtemp=0$, the response is pure noise---a fair coin flip---with no dependence on the covariates.
When $\invtemp = +\infty$, the response is fully determined by a homogeneous linear classifier, simply denoting the side of the hyperplane $\set{ x \in \R^d : x^\T\param = 0 }$ that $\bfx$ lies on.

We assume the observed data $(\bfx_1,\bfy_1),\dotsc,(\bfx_n,\bfy_n)$ are independent copies of $(\bfx,\bfy)$, with distribution determined by the unknown parameter vector $\truth \in \sphere$ (and inverse temperature $\invtemp$).
For a given $\epsilon \in \intoo{0,1}$, the goal is to find an estimate $\estimate = \estimate((\bfx_i,\bfy_i)_{i=1}^n) \in \sphere$ based on these data (and possibly also $\invtemp$) such that $\estimate$ satisfies
\begin{equation*}
  \norm{\estimate - \truth} \leq \epsilon
\end{equation*}
(either in expectation or with high probability over the realization of the data).

The \emph{sample complexity} $n^\star(d,\invtemp,\epsilon)$ is the smallest sample size $n$ such that the above task is achievable by some estimator. 
Intuitively, the estimation error for the maximum likelihood estimator here is asymptotically normal; leading one to believe that the sample complexity must scale as $d/\epsilon^2$.
However, as $\invtemp$ increases to $+\infty$, the estimation problem gradually becomes a noiseless halfspace learning problem.
The sample complexity for the latter problem is known to be $d/\epsilon$.
This means that the inverse temperature $\invtemp$ should play a crucial role in the sample complexity; going from small $\invtemp$ to large $\invtemp$, the sample complexity curve must have one or more change-points.
In this paper, we coarsely determine this dependence on $\invtemp$.
In particular, we show that, up to logarithmic factors in $d$ and $1/\epsilon$ in the expressions below, the sample complexity satisfies
\begin{equation*}
  n^\star(d,\invtemp,\epsilon)
  \asymp
  \begin{cases}
    \displaystyle
    \frac{d}{\invtemp^2\epsilon^2} & \text{if $\invtemp \lesssim 1$ (high temperatures)} ;
    \vspace{1mm}
    \\
    \displaystyle
    \frac{d}{\invtemp\epsilon^2} & \text{if $1 \lesssim \invtemp \lesssim 1/\epsilon$ (moderate temperatures)} ;
    \vspace{1mm}
    \\
    \displaystyle
    \frac{d}{\epsilon} & \text{if $\invtemp \gtrsim 1/\epsilon$ (low temperatures)} .
  \end{cases}
\end{equation*}

\subsection{Motivation}
\label{sec:motivation}

Our original motivation for studying this problem comes from the application to noisy one-bit (compressive) sensing, in which only a single bit is retained per linear measurement of a signal~\citep{boufounos20081}.
In that context, \citet{plan2012robust} proposed a robust linear estimator that is well-behaved under a variety of observation models, including the logistic regression model that we consider.
This estimator (which in the present context is essentially the same as the ``Average'' algorithm of \citet{servedio1999pac}) was shown by \citet{plan2017high} to have sample complexity
\begin{equation*}
  O\del*{ \frac{d}{\min\set{\invtemp^2,1}\epsilon^2} } ,
\end{equation*}
improving on an earlier analysis of \citet{plan2012robust} that had a worse dependence on $\epsilon$; note that here we do not assume that $\truth$ is sparse.
The optimality of this estimator in the high-temperature regime ($\invtemp \lesssim 1$) is readily established as a standard application of Fano's inequality \citep[see, e.g.,][Appendix C.1]{chen2016bayes}.
However, it was unclear whether this sample complexity is optimal in other regimes.
In particular, in the zero-temperature regime (i.e., $\invtemp = +\infty$), the $\bfy_i$ are determined by the sign of $\bfx_i^\T\truth$, so the problem becomes equivalent to that of PAC learning homogeneous linear classifiers under spherically symmetric distributions on $\bfx$.
For that problem, the sample complexity is $\Theta(d/\epsilon)$, as established by \citet{long1995sample,long2003upper}; note that the dependence on $\epsilon$ is considerably reduced.\footnote{%
  \citet{servedio1999pac} studied the ``Average'' algorithm in this context, allowing for the possibility that each observed label is independently flipped with some fixed probability $\eta \in \intco{0,1/2}$, and obtains a sample complexity upper bound of $O(d/((1-2\eta)\epsilon)^2)$.%
}
\citet{jacques2013robust} also gives a statement of the lower bound in the context of one-bit compressive sensing.

Lower bounds on the sample complexity in our setting do not directly follow from standard lower bounds from statistical learning theory for homogeneous linear classifiers~\citep[e.g.,][]{devroye1995lower}, as the distributions exhibiting the lower bounds generally do not conform to our statistical model.
In particular, the support of $\bfx$ in these lower bounds is typically taken to be a shattered finite set of points, and the conditional distribution of $\bfy$ given $\bfx$ may not be of the form in logistic regression.
The exceptions, as mentioned above, are those based on Fano's inequality for the high temperature regime, and the lower bound in the zero temperature regime of~\citet{long1995sample}.

Our upper and lower bounds resolve the dependence of the sample complexity on the inverse temperature (up to logarithmic factors in $1/\epsilon$ and $d$), particularly in the regime where $1 \lesssim \invtemp < +\infty$.

\subsection{Techniques}
\label{sec:techniques}

\paragraph{Lower bounds.}

Our lower bounds for the high and moderate temperature regimes (\Cref{thm:moderatetemp}) are proved using Fano's inequality; for convenience, we use a Bayesian version due to \citet{zhang2006information}, although the ``Generalized Fano'' approach of \citet{han1994generalizing} would also work.
The bound in the high temperature regime is a ``textbook'' application that uses a uniform quadratic bound on the Kullback–Leibler (KL) divergence between the distributions determined by nearby parameter vectors.
However, the moderate temperature regime requires a more refined analysis that does not appear to be standard.
To facilitate the required computation, we use the Bregman divergence form of the KL divergence between Bernoulli distributions.

For the low temperature regime, this version of Fano's inequality cannot be used, as the aforementioned KL divergence becomes unbounded.
The basic form of Fano's inequality is applicable, as the conditional entropy of the (randomly chosen) parameter given the data can be estimated.
However, we were not able to obtain the optimal lower bound this way in this low temperature regime.
Instead, we replicate the combinatorial argument of \citet{long1995sample} with modifications to handle finite $\invtemp$ (\Cref{thm:lowtemp}).

\paragraph{Upper bounds.}

To establish upper bounds on the sample complexity, it is natural to consider the maximum likelihood estimator (MLE), which is equivalent to finding $\param \in \sphere$ that minimizes the empirical risk with respect to the logistic loss: $(1/n) \sum_{i=1}^n  g(-\invtemp \bfy_i\bfx_i^\T\param)$.
The convexity of the logistic loss potentially makes the MLE computationally tractable (perhaps after extending the parameter space to the ball).
The risk of a given $\param$ is the expected value of this empirical risk, and the excess risk relative to that of $\truth$ is the KL divergence between the distributions determined by $\param$ and $\truth$.
This KL divergence can, in turn, be related to the parameter error $\norm{\param - \truth}$ using our analysis from the lower bound.
However, bounding the excess risk sharply enough appears to be challenging.
A standard approach in statistical learning theory is to use techniques like Rademacher averages from the theory of empirical processes to relate excess risks to the excess empirical risks.
Unfortunately, using such tools designed for smooth loss functions like the logistic loss~\citep[e.g.,][]{srebro2010smoothness} leads to a sample complexity with suboptimal dependences on $\epsilon$ and $\invtemp$.
The distribution of $(\bfx,\bfy)$ satisfies the Tsybakov-Mammen margin condition $\Pr[ \abs{ g'(\invtemp \bfx^\T\truth) - 1/2 } \leq t ] \leq C_0 t^{\alpha/(1-\alpha)}$ with $C_0 = O(1/\invtemp)$ and $\alpha = 1/2$~\citep{mammen1999smooth}, but this only improves excess classification error bounds, not parameter estimation error.

A different approach is to directly analyze the MLE by computing tight Taylor expansions of the estimation error, and using properties such as self-concordance of the logistic function, in a way that keeps track of the dimension dependence~\citep[e.g.,][]{he2000parameters,portnoy1988asymptotic,bach2010self,ostrovskii2021finite}, and, ideally, the inverse temperature.
These analyses of the MLE match the performance guarantees in the leading order terms as predicted by classical asymptotic analysis.
However, they essentially treat $\invtemp$ as a constant (resulting in suboptimal dependence on $\invtemp$ in lower order terms), which we cannot afford to do as the moderate and low temperature regimes are defined by the comparison of $\invtemp$ to $1/\epsilon$.
Redoing the analysis entirely in our setting boils down to showing $(\param - \truth)^\T \nabla L(\param) > 0$ for all $\param$ that are $\epsilon$ away from the true parameter vector $\truth$, with $L$ being the negative log-likelihood function.
At this point, we need a tight approximation of $\nabla L(\param)$ around $\nabla L(\truth)$, which one may hope to obtain using the self-concordance property.
However, in our setting, the inner product depends on $\param$, and a uniform bound does not give the optimal scaling.

A very recent work of \citet{kuchelmeister2023finite} studies the estimation of the parameter vector in a probit regression model (under a normal design) by maximizing the likelihood under a (misspecified) logistic regression model.
They directly use techniques from empirical process theory specialized to the normal design to provide upper bounds on the sample complexity needed to estimate $\truth$ up to a given $\ell_2$ error (under certain assumptions on the signal-to-noise ratio in the probit model; they are also concerned with estimation of the signal-to-noise ratio itself, which is beyond the scope of our work).
We leverage essential parts of their analysis to establish our sample complexity upper bound in the moderate and low temperature regimes of our problem, which requires new moment bounds in the logistic regression observation model.
We also give sample complexity bounds based on empirical risk minimization (for zero-one loss) in the low temperature regime, mostly using standard techniques.
As mentioned above, an optimal estimator for the high temperature regime was already known~\citep{plan2017high}.

Some of estimators we study are applicable under other designs, but the analyses in this work make heavy use of symmetry properties of the normal distribution.
Bounds on various Gaussian integrals essential in our proofs are collected in \Cref{sec:integrals}.

\subsection{Other related works}
\label{sec:related}

\paragraph{Improper learning.}
Several prior works analyze ``improper learning'' algorithms for (possibly misspecified) logistic regression~\citep[e.g.,][]{kakade2004online,zhang2006information,hazan2014logistic,foster2018logistic,mourtada2022improper} that do not necessarily produce an estimate of $\truth$, which may not be sensible anyway if the model is misspecified.
Instead, the goal of these algorithms is to achieve low prediction error guarantees.
The lower bound of \citet{hazan2014logistic} exploits misspecification to show that larger parameter norm is more detrimental to proper online learning than it is to improper online learning.
In our well-specified setting for parameter estimation, the parameter norm has a very different effect.

\paragraph{High-dimensional proportional asymptotic analysis.}
Another line of work considers the performance of MLE and regularized variants in the proportional asymptotic regime, where both $d$ and $n$ increase to infinity with $d/n$ tending to a constant $\delta > 0$.
In this setup, \citet{sur2019modern} are able to precisely characterize the region in the plane of $\delta$ and $\invtemp$ where the MLE exists.
\citet{salehi2019impact} consider the various regularized variants of MLE and characterize their asymptotic performance.
Neither work directly reveals the dependence of the sample complexity on $\invtemp$.

\subsection{Notations}

For notational convenience, we assume that a Bernoulli distribution $\Bernoulli(p)$ has support on $\set{-1,1}$, with the ``mean parameter'' $p$ being the probability of $1$.
We occasionally associate each $\param \in \sphere$ with a homogeneous linear classifier $h_\param \colon \R^d \to \set{-1,1}$, given by $h_\param(x) = 1$ if $x^\T\param > 0$ and $h_\param(x) = -1$ otherwise.
For any $\param, \param' \in \sphere$, let $\err_\param(\param') = \Pr(h_{\param'}(\bfx) \neq \bfy) = \Pr(\bfy\bfx^\T\param'\leq0)$ be the error rate of $h_{\param'}$ when the distribution of $(\bfx,\bfy)$ is specified by parameter $\param$.

\section{Lower bounds on the sample complexity}

In this section, we give two lower bounds on the sample complexity.

\subsection{Moderate and high temperatures}

The following \namecref{thm:moderatetemp} establishes our sample complexity lower bound for moderate and high temperatures.

\begin{theorem}
  \label[theorem]{thm:moderatetemp}
  Fix $\epsilon \in \intoo{0,1}$.
  Suppose $\estimate$ is an estimator that, for any $\truth \in \sphere$,
  \begin{equation*}
    \E\sbr*{ \norm{\estimate((\bfx_i,\bfy_i)_{i=1}^n) - \truth} } \leq \epsilon ,
  \end{equation*}
  where the expectation is with respect to the data with distribution determined by $\truth$.
  Then the sample size $n$ must satisfy
  \begin{equation*}
    n \geq \frac{(d-1)\ln2 - \ln4}{32\epsilon^2\invtemp\min\set{\invtemp, 2\sqrt{2/\pi}}} .
  \end{equation*}
\end{theorem}

The proof of \Cref{thm:moderatetemp} uses the following information theoretic lower bound method due to \citet{zhang2006information}.
\begin{lemma}[Theorem 6.1 of~\citealp{zhang2006information}]
  \label[lemma]{lem:fano}
  Let $\prior$ be a probability measure on a parameter space $\Params$ indexing a family of probability measures $(P_{\param})_{\param \in \Params}$ on a data space $\calZ$, and let $L \colon \Params \times \Params \to \R$ be a loss function.
  Let $\bfparam \sim \prior$ and $\bfZ \mid \bfparam \sim P_{\bfparam}$.
  For any (possibly randomized) estimator $\estimate \colon \calZ \to \Params$,
  \begin{equation*}
    \E\sbr*{
      L(\bftheta,\estimate(\bfZ))
    }
    \geq \frac12
    \sup\set*{
      \varepsilon : \inf_{\param \in \Params} -\ln(\prior(B(\param,\varepsilon)))
      \geq 2\klbound + \ln 4
    }
  \end{equation*}
  where $B(\param,\varepsilon) = \set{ \param' \in \Params : L(\param,\param') < \varepsilon }$ and $\klbound = \E_{(\bfparam,\bfparam') \sim \prior \otimes \prior}\sbr*{\KL(P_{\bfparam} \| P_{\bfparam'})}$.
\end{lemma}

\begin{proof}[Proof of \Cref{thm:moderatetemp}]
  We prove the contrapositive.
  Assume that
  \begin{equation*}
    n < \frac{(d-1)\ln2 - \ln4}{32\epsilon^2\invtemp\min\set{\invtemp, 2\sqrt{2/\pi}}} .
  \end{equation*}
  Set $\fano = 4\epsilon$.
  Fix a unit vector $u \in \sphere$, and let $C = \set{ \param \in \sphere : \norm{\param - u} \leq \fano }$ be the spherical cap of radius $\fano$ around $u$.
 Let $\Params = \sphere$, $\prior$ be the uniform measure on $C$, and $L(\param,\param') = \norm{\param - \param'}$.
  For each $\param \in \Params$, we let $P_{\param}$ denote the joint distribution of the data $\bfZ = ((\bfx_1,\bfy_1),\dotsc,(\bfx_n,\bfy_n))$ as determined by our model with parameter $\param$.
  Observe that $\prior(B(\param,\varepsilon)) \leq (\varepsilon/\fano)^{d-1}$ for any $\varepsilon \leq \fano$.
  Therefore, provided that
  \begin{equation*}
    (d-1)\ln2 > 2\klbound + \ln4 ,
  \end{equation*}
  we have $\E\sbr{ \norm{\bfparam - \estimate(\bfZ)} } > \fano/4 = \epsilon$ by \Cref{lem:fano}.
  So it remains to establish the above displayed inequality.

  Since the $n$ data are i.i.d.~in each of $P_{\param}$ and $P_{\param'}$, and the marginal distribution of $\bfx$ is the same in both $P_{\param}$ and $P_{\param'}$, the chain rule for KL divergence implies
  \begin{equation*}
    \KL(P_{\param} \| P_{\param'})
    = n \E\sbr*{
      \KL(\Bernoulli(g'(\invtemp\bfx^\T\param)) \| \Bernoulli(g'(\invtemp\bfx^\T\param')))
    }
    .
  \end{equation*}
  Let $\bfz = \bfx^\T\param$ and $\bfz' = \bfx^\T\param'$, so each of $\bfz$ and $\bfz'$ is a standard normal random variable, and the correlation $\rho$ between $\bfz$ and $\bfz'$ satisfies
  \begin{equation*}
    \rho
    = \param^\T\param'
    = 1 - \frac12\norm{\param - \param'}^2
    \geq 1 - \frac12(2\fano)^2 = 1 - 2\fano^2 ,
  \end{equation*}
  where we have used the triangle inequality $\norm{\param - \param'} \leq \norm{\param - u} + \norm{\param' - u} \leq 2\fano$.
  By \Cref{lem:klbound} (see \Cref{sec:integrals}),
  \begin{align*}
    \E\sbr*{ \KL(\Bernoulli(g'(\invtemp\bfx^\T\param)) \| \Bernoulli(g'(\invtemp\bfx^\T\param'))) }
    & = \E\sbr*{ \KL(\Bernoulli(g'(\invtemp\bfz)) \| \Bernoulli(g'(\invtemp\bfz'))) } \\
    & \leq \frac{\invtemp}{2} \del{1-\rho} \min\set*{ \invtemp, 2\sqrt{2/\pi} } \\
    & \leq \fano^2 \invtemp \min\set*{ \invtemp, 2\sqrt{2/\pi} }
    .
  \end{align*}
  Therefore
  \begin{equation*}
    2\klbound + \ln4
    \leq 2n\fano^2\invtemp\min\set{ \invtemp, 2\sqrt{2/\pi} } + \ln4
    < (d-1)\ln2
    .
    \qedhere
  \end{equation*}
\end{proof}

\subsection{Low temperatures}

The sample complexity lower bound from \Cref{thm:moderatetemp} tends to $0$ as $\invtemp \to \infty$.
This appears to be a limitation of the proof technique, which is only useful for $\invtemp \lesssim 1/\epsilon$.
We next establish a lower bound that improves on \Cref{thm:moderatetemp} for $\invtemp \gg 1/\epsilon$.

\begin{theorem}
  \label[theorem]{thm:lowtemp}
  Fix $\epsilon \in \intoo{0,1}$.
  Assume $\invtemp \geq 4\sqrt{2/\pi}/\epsilon$.
  Suppose $\estimate$ is an estimator that, for any $\truth \in \sphere$,
  \begin{equation*}
    \Pr\del{ \norm{\estimate((\bfx_i,\bfy_i)_{i=1}^n) - \truth} < \epsilon } \geq \frac12 ,
  \end{equation*}
  where the probability is with respect to the data with distribution determined by $\truth$.
  Then the sample size $n$ must satisfy
  \begin{equation*}
    n \geq \frac{d-1}{\epsilon} \cdot \frac{\log(\log(2e/\epsilon)) - \log(16e)}{(1+o(1))\log(2e/\epsilon)}
  \end{equation*}
  where the $o(1)$ term depends only $\epsilon$ and vanishes as $\epsilon \to 0$.
  If $\invtemp = \infty$, then the sample size must, in fact, satisfy $n \geq (d-1)/(8e\epsilon)$.
\end{theorem}

The proof of \Cref{thm:lowtemp} essentially follows that of \citet{long1995sample} for a lower bound on the sample complexity of PAC learning homogeneous linear classifiers under the uniform distribution on $\sphere$.
In \citeauthor{long1995sample}'s setting, the data distribution coincides with ours for $\invtemp = \infty$.
We make a minor modification to the argument to also handle finite (but large) $\invtemp$.
The extra $\log\log(1/\epsilon)/\log(1/\epsilon)$ factor in the finite $\invtemp$ case appears to be an artifact of the proof technique.
We also note that the $\invtemp = \infty$ case is directly implied by the main result of \citet{long1995sample}.

\citeauthor{long1995sample}'s proof, as well as ours, relies on the following bound on the shattering number of homogeneous linear classifiers.
\begin{lemma}[\citealp{winder1966partitions}, Corollary on page 816]
  \label[lemma]{lem:shatter}
  Let $H$ be the family of homogeneous linear classifiers in $\R^d$.
  For any $x_1,\dotsc,x_n \in \R^d$,
  \begin{equation*}
    \card{\set{ (h(x_1),\dotsc,h(x_n)) : h \in H }} \leq 2\del*{ \frac{ne}{d-1} }^{d-1} .
  \end{equation*}
\end{lemma}

We also need the following well-known lower bound on the packing number for $\sphere$; the proof is given for completeness.
\begin{lemma}
  \label[lemma]{lem:packing}
  There exists an $\epsilon$-packing of the unit sphere $\sphere$ with respect to $\ell_2$ distance of cardinality $(1/\epsilon)^{d-1}$.
\end{lemma}
\begin{proof}
  Let $W$ be a subspace of $d$-dimensional Euclidean space of dimension $d-1$.
  By a standard volume argument, there is an $\epsilon$-packing $p_1,\dotsc,p_M$ of the unit ball in $W$ with $M \geq (1/\epsilon)^{d-1}$.
  For each point $p_i$, there is a corresponding point $p_i' \in \sphere$ whose orthogonal projection to $W$ is $p_i$.
  For any $i \neq j$, we have $\norm{p_i' - p_j'} \geq \norm{p_i - p_j} \geq \epsilon$, so $p_1',\dotsc,p_M'$ is an $\epsilon$-packing of $\sphere$.
\end{proof}

Finally, we need the following bound on the error rate of the homogeneous linear classifier $h_{\truth}$.
(The \namecref{lem:opt} holds for all $\invtemp\geq0$, not just $\invtemp\gtrsim1/\epsilon$.)
\begin{lemma}
  \label[lemma]{lem:opt}
  The error rate of $h_{\truth}$ satisfies
  \begin{equation*}
    \err_{\truth}(\truth)
    \leq
    \frac1\invtemp \sqrt{\frac2\pi}
    .
  \end{equation*}
\end{lemma}
\begin{proof}
  Since
  \begin{equation*}
    \err_{\truth}(\truth)
    = \Pr(\bfy\bfx^\T\truth \leq 0)
    = \E\sbr{g'(-\invtemp\abs{\bfx^\T\truth})}
    = \E\sbr{g'(-\invtemp\abs{\bfz})}
  \end{equation*}
  for a normal random variable $\bfz$, the claim now follows from \Cref{lem:g'abs} (see \Cref{sec:integrals}).
\end{proof}

\begin{proof}[Proof of \Cref{thm:lowtemp}]
  Assume without loss of generality that $\epsilon \leq 1/2$.
  Let $\bfx_1,\dotsc,\bfx_n$ be independent standard normal random vectors.
  Let $\bfy_1^\param,\dotsc,\bfy_n^\param$ for all $\param \in \sphere$ be independent Bernoulli random variables with $\bfy_i^\param \mid (\bfx_1,\dotsc,\bfx_n) \sim \Bernoulli(g'(\invtemp\bfx_i^\T\param))$, and let $\bfZ^\param = ((\bfx_1,\bfy_1^\param),\dotsc,(\bfx_n,\bfy_n^\param))$.
  So the data $\bfZ^\param$ follows our model with parameter $\param$, but all $\set{ \bfZ^\param : \param \in \sphere }$ share the same $\bfx_1,\dotsc,\bfx_n$.
  Let $\emperr_\param(\param) = (1/n) \sum_{i=1}^n \ind{h_\param(\bfx_i) \neq \bfy_i^\param}$ be the empirical error rate of $h_\param$ (with respect to $\bfZ^\param$).

  Suppose $\Pr(\norm{\estimate(\bfZ^\param) - \param} < \epsilon) \geq 1/2$ for all $\param \in \sphere$.
  By \Cref{lem:opt}, since $\invtemp \geq 4\sqrt{2/\pi}/\epsilon$, we have
  \begin{equation*}
    \err_\param(\param) \leq \frac{\epsilon}4 .
  \end{equation*}
  Let $U$ be a $2\epsilon$-packing of $\sphere$ with respect to $\ell_2$ of cardinality $\card{U} \geq (2\epsilon)^{-(d-1)}$, as guaranteed to exist by \Cref{lem:packing}.
  Let $G_\param$ be the (indicator of) the event $\norm{\estimate(\bfZ^\param) - \param} < \epsilon$ and $\emperr_\param(\param) \leq \epsilon$.
  By Markov's inequality,
  \begin{equation*}
    \Pr(\emperr_\param(\param) \geq \epsilon) \leq \frac14 .
  \end{equation*}
  So
  \begin{equation*}
    \E\sbr*{ \sum_{\param \in U} G_\param }
    =
    \sum_{\param \in U} \E\sbr{ G_\param }
    \geq
    \sum_{\param \in U} \del*{ 1 - \frac12 - \frac14 }
    = \frac{\card{U}}4
    \geq \frac14 \del*{ \frac{1}{2\epsilon} }^{d-1} .
  \end{equation*}
  Consider any two $\param, \param' \in U$.
  If $\bfZ^\param = \bfZ^{\param'}$, then either $\norm{\estimate(\bfZ^\param) - \param} > \epsilon$ or $\norm{\estimate(\bfZ^{\param'}) - \param'} > \epsilon$ (since $\norm{\param - \param'} \geq 2\epsilon$), so at least one of $G_\param$ and $G_{\param'}$ is zero.
  Moreover, if $G_\param = 1$, then the labels $\bfy_i^\param$ are realized by the homogeneous linear classifier determined by $\param$---except for up to $\floor{n\epsilon}$ labels, which could be flipped.
  Therefore $\E\sbr*{ \sum_{\param \in U} G_\param }$ is at most the number of ways of labeling $(\bfx_1,\dotsc,\bfx_n)$ by homogeneous linear classifiers determined by weight vectors from $U$, multiplied by $2^{\floor{n\epsilon}} \binom{n}{\floor{n\epsilon}}$.
  Hence, by \Cref{lem:shatter},
  \begin{equation*}
    \E\sbr*{ \sum_{\param \in U} G_\param }
    \leq 2\del*{ \frac{ne}{d-1} }^{d-1} \cdot \del*{ \frac{2ne}{n\epsilon} }^{n\epsilon}
    = 2\del*{ \frac{e}{M\epsilon} \cdot \del*{ \frac{2e}{\epsilon} }^{1/M} }^{d-1}
  \end{equation*}
  where $M = (d-1)/(n\epsilon)$.
  Combining the upper and lower bounds on $\E\sbr*{ \sum_{\param \in U} G_\param }$ gives
  \begin{equation*}
    \frac{e}{M\epsilon} \cdot \del*{ \frac{2e}{\epsilon} }^{1/M}
    \geq \frac1{8^{1/(d-1)}} \cdot \frac{1}{2\epsilon}
    \geq \frac1{16\epsilon} .
  \end{equation*}
  Taking logarithms of both sides and simplifying gives the following inequality:
  \begin{equation*}
    \del*{\frac{M}{16e}} \log \del*{\frac{M}{16e}} \leq \frac{1}{16e} \log \frac{2e}{\epsilon} .
  \end{equation*}
  Let $T(\epsilon) = \log(2e/\epsilon)/(16e)$.
  Then, using asymptotic expansion of the product log function~\citep{corless1996lambert}, we have
  \begin{equation*}
    \frac{M}{16e} \leq \frac{T(\epsilon)}{\log T(\epsilon)}(1 + o(1)) ,
  \end{equation*}
  where the $o(1)$ term vanishes as $T(\epsilon) \to \infty$ (i.e., $\epsilon \to 0$).
  This implies
  \begin{equation*}
    n \geq \frac{d-1}{16e\epsilon} \cdot \frac{\log T(\epsilon)}{(1+o(1))T(\epsilon)}
  \end{equation*}
  as claimed.

  If $\invtemp = \infty$, then $\Pr(\emperr_\param(\param) = 0) = 1$, and hence we have
  \begin{equation*}
    \frac12 \del*{ \frac{1}{2\epsilon} }^{d-1}
    \leq \E\sbr*{ \sum_{\param \in U} G_\param }
    \leq 2\del*{ \frac{ne}{d-1} }^{d-1} .
  \end{equation*}
  Therefore
  \begin{equation*}
    n \geq \frac{d-1}{8e\epsilon}
    .
    \qedhere
  \end{equation*}
\end{proof}

\section{Upper bounds on the sample complexity}

In this section, we give upper bounds on the sample complexity based on three different estimators for the three different regimes of $\invtemp$.
Throughout this section, we fix a ``ground truth'' parameter $\truth \in \sphere$ determining the distribution of $(\bfx,\bfy)$.

\subsection{High temperatures}
\label{sec:hightempupper}

The analysis of \citet[Corollary 3.5]{plan2017high} implies, for any $\epsilon \in \intoo{0,1}$, the ``linear estimator'' $\linest$ of \citet{plan2012robust} (or the ``Average'' algorithm of \citet{servedio1999pac})
\begin{equation*}
  \linest((\bfx_i,\bfy_i)_{i=1}^n) = \argmax_{\param \in \sphere} \frac1n \sum_{i=1}^n \bfy_i\bfx_i^\T\param
\end{equation*}
satisfies $\norm{\linest((\bfx_i,\bfy_i)_{i=1}^n) - \truth} \leq \epsilon$ in expectation provided that
\begin{equation}
  \label{eq:hightempupper}
  n \geq C \max\set*{ \frac1{\invtemp^2} , 1 } \frac{d}{\epsilon^2} ,
\end{equation}
where $C>0$ is a universal positive constant.
This matches the lower bound from \Cref{thm:moderatetemp} (up to constants) when $\invtemp \lesssim 1$.

\subsection{Moderate and low temperatures}
\label{sec:moderatetempupper}

When $\invtemp \gtrsim 1$, the sample size requirement of the linear estimator in \Cref{eq:hightempupper} has a suboptimal dependence on $\invtemp$.
In particular, the sample size requirement does not become smaller as $\invtemp$ becomes larger.

In this section, we analyze a different estimator, the empirical ReLU risk minimizer $\reluest$:
\begin{equation}
  \label{eq:reluest}
  \reluest((\bfx_i,\bfy_i)_{i=1}^n) \in \argmin_{\param \in \sphere} \frac1n \sum_{i=1}^n \sbr{-\bfy_i\bfx_i^\T\param}_+ ,
\end{equation}
where $\sbr{x}_+ = \max\set{0,x}$ (the rectified linear function used in ReLUs).
Notice that if $\sbr{\cdot}_+$ was omitted in each summand of \Cref{eq:reluest}, then the estimator $\reluest$ would be the same as $\linest$.
This estimator is also related to the Perceptron algorithm~\citep{rosenblatt1958perceptron}, since the latter can be viewed as a stochastic subgradient method for minimizing the empirical ReLU risk.
However, we do not know if such subgradient methods minimize the objective over the nonconvex domain $\sphere$.

\begin{theorem}
  \label[theorem]{thm:moderatetempupper}
  Assume $\invtemp \geq 1+c$ for some positive constant $c>0$.
  Fix any $\epsilon, \delta \in \intoo{0,1}$, and suppose
  \begin{equation*}
    n \geq C \del*{
      \frac{d\log(d/\epsilon) + \log(1/\delta)}{\invtemp \epsilon^2}
      + \frac{d\log(d/\epsilon) + \log(1/\delta)}{\epsilon}
    }
  \end{equation*}
  where $C>0$ is an absolute constant.
  Then with probability at least $1-\delta$, we have
  \begin{equation*}
    \norm{\reluest((\bfx_i,\bfy_i)_{i=1}^n) - \truth} \leq \epsilon .
  \end{equation*}
\end{theorem}

The main technical work in the proof of \Cref{thm:moderatetempupper} is understanding the concentration properties of random variables of the form
\begin{equation*}
  \bfdelta_{\param} := \sbr{ -\bfy \bfx^\T \param }_+ - \sbr{ -\bfy \bfx^\T \truth }_+ , \quad \param \in \sphere .
\end{equation*}
This was studied by \citet{kuchelmeister2023finite} in the case of the probit regression model.
We obtain the necessary moment bounds for the logistic regression model.
The following \namecref{lem:relumoments} (proved in \Cref{sec:relumomentsproof}) gives the required bounds.

\begin{lemma}
  \label[lemma]{lem:relumoments}
  For any $\param \in \sphere$, and any integer $q \geq 2$,
  \begin{equation*}
    \E \sbr*{ \abs{\bfdelta_{\param}}^q }
    \leq \frac{q!v}{2} b^{q-2}
  \end{equation*}
  where
  \begin{equation*}
    b = C \norm{\param - \truth}
    \quad \text{and} \quad
    v \leq C \norm{\param - \truth}^2 \del*{ \frac1{\invtemp} + \norm{\param - \truth} }
  \end{equation*}
  and $C>0$ is an absolute constant.
\end{lemma}

Given \Cref{lem:relumoments}, the rest of the analysis is mostly standard.
\Cref{lem:relumoments} and Bernstein's inequality implies that each random variable in the empirical process above is concentrated around its expectation.
\begin{lemma}
  \label[lemma]{lem:reluconc}
  For any $\param \in \sphere$, define $\bfdelta_{\param}^i := \sbr{-\bfy_i\bfx_i^\T\param}_+ - \sbr{-\bfy_i\bfx_i^\T\truth}_+$ for all $i=1,\dotsc,n$.
  For any $\delta \in \intoo{0,1}$, with probability at least $1-\delta$, we have
  \begin{equation*}
    \E\sbr*{ \bfdelta_{\param} } - \frac1n \sum_{i=1}^n \bfdelta_{\param}^i
    \leq
    C \norm{\param - \truth} \sqrt{\del*{ \frac1{\invtemp} + \norm{\param - \truth} } \frac{\log(1/\delta)}{n}} + C \norm{\param - \truth} \frac{\log(1/\delta)}{n} ,
  \end{equation*}
  where $C>0$ is an absolute constant.
\end{lemma}

To bound all random variables in the stochastic process simultaneously, we use a covering argument, again following \citet{kuchelmeister2023finite}.
\begin{lemma}[\citealp{kuchelmeister2023finite}, Lemma 5.1.2]
  \label[lemma]{lem:covering}
  For any $\varepsilon_0\in\intoo{0,1}$ and any $A \subseteq \sphere$, let $T(A,\varepsilon_0)$ be an $\varepsilon_0$-cover of $A$ with respect to $\ell_2$ distance.
  Define $\bfdelta_{\param}^i := \sbr{-\bfy_i\bfx_i^\T\param}_+ - \sbr{-\bfy_i\bfx_i^\T\truth}_+$ for all $i=1,\dotsc,n$ and all $\param \in A$.
  For any $\delta \in \intoo{0,1}$, with probability at least $1-\delta$, we have
  \begin{equation*}
    \sup_{\param \in A}
    \cbr*{ \E\sbr*{ \bfdelta_{\param} } - \frac1n \sum_{i=1}^n \bfdelta_{\param}^i }
    \leq
    \max_{\param \in T(A,\varepsilon_0)}
    \cbr*{ \E\sbr*{ \bfdelta_{\param} } - \frac1n \sum_{i=1}^n \bfdelta_{\param}^i }
    + \varepsilon_0 \del*{ \sqrt{\frac{2\ln(1/\delta)}{n}} + 2\sqrt{d} } .
  \end{equation*}
\end{lemma}

To relate $\E\sbr{\bfdelta_{\param}}$ to $\norm{\param - \truth}$, we use the following \namecref{lem:paramdiff}.
\begin{lemma}
  \label[lemma]{lem:paramdiff}
  For any $\param \in \sphere$, we have
  \begin{equation*}
    \E\sbr*{ \bfdelta_{\param} }
    \geq \frac18\sqrt{\frac2\pi} \del*{1 - \frac1{\invtemp^2}} \norm{ \param - \truth }^2 .
  \end{equation*}
\end{lemma}
\begin{proof}
  The expected KL divergence of $\Bernoulli(g'(\invtemp\bfx^\T\param))$ from $\Bernoulli(g'(\invtemp\bfx^\T\truth))$ is the expected excess logistic loss of $w := \invtemp\param$ compared to $w^\star := \invtemp\truth$:
  \begin{align*}
    \E\sbr*{\KL(\Bernoulli(g'(\invtemp\bfx^\T\param)) \| \Bernoulli(g'(\invtemp\bfx^\T\truth)))}
    & = \E\sbr*{\ln(1 + e^{-\bfy \bfx^\T w})}
    - \E\sbr*{\ln(1 + e^{-\bfy \bfx^\T w^\star})}
  \end{align*}
  (recalling that the conditional distribution of $\bfy$ given $\bfx=x$ is $\Bernoulli(g'(x^\T w^\star))$).
  Furthermore, the logistic loss $w \mapsto \ln(1 + e^{-\bfy\bfx^\T w})$ decomposes into the sum of a label-dependent part and a label-independent part:
  \begin{equation*}
    \ln(1 + e^{-\bfy\bfx^\T w})
    = \sbr*{ -\bfy\bfx^\T w }_+ + \ln(1 + e^{-\abs{\bfx^\T w}}) .
  \end{equation*}
  (This decomposition is the same as that from \citet{kuchelmeister2023finite}, and a similar decomposition was used by \citet{bach2010self}.)
  We can thus write the excess expected logistic loss as
  \begin{align*}
    \E\sbr*{\ln(1 + e^{-\bfy \bfx^\T w})}
    - \E\sbr*{\ln(1 + e^{-\bfy \bfx^\T w^\star})}
    & =
    \E\sbr*{ \sbr*{ -\bfy\bfx^\T w }_+ } + \E\sbr*{ \ln(1 + e^{-\abs{\bfx^\T w}}) }
    \\
    & \qquad
    - \E\sbr*{ \sbr*{ -\bfy\bfx^\T w^\star }_+ } - \E\sbr*{ \ln(1 + e^{-\abs{\bfx^\T w^\star}}) }
    \\
    & = 
    \E\sbr*{ \sbr*{ -\bfy\bfx^\T w }_+ - \sbr*{ -\bfy\bfx^\T w^\star }_+ }
  \end{align*}
  since $\bfx^\T w$ and $\bfx^\T w^\star$ have the same distribution.
  Therefore, we have
  \begin{equation*}
    \E\sbr*{ \sbr*{ -\bfy\bfx^\T\param }_+ - \sbr*{ -\bfy\bfx^\T\truth }_+ }
    = \frac1\invtemp 
    \E\sbr*{\KL(\Bernoulli(g'(\invtemp\bfx^\T\param)) \| \Bernoulli(g'(\invtemp\bfx^\T\truth)))}
    .
  \end{equation*}
  The claim now follows by applying \Cref{lem:klbound} (see \Cref{sec:integrals}).
\end{proof}

\begin{proof}[Proof of \Cref{thm:moderatetempupper}]
  Fix some $\varepsilon_0 , r_0 \in \intoo{0,1}$ and let $r_j = 2^j r_0$ for $1 \leq j \leq J := \ceil{\log_2(2/r_0)}$.
  Define $A_0 = \set{ \param \in \sphere : \norm{\param - \truth} \leq r_0 }$ and $A_j = \set{ \param \in \sphere : r_{j-1} < \norm{\param - \truth} \leq r_j }$ for all $j \geq 1$.
  Let $T_j = T(A_j,\varepsilon_0(r_j/2)^{3/2})$ be an $\varepsilon_0(r_j/2)^{3/2}$-cover of $A_j$ with respect to $\ell_2$ distance of cardinality at most $(2(2/r_j)^{3/2}\varepsilon_0^{-1})^{d-1}$ (cf.~\Cref{lem:packing}).
  Fix $\delta_0 \in \intoo{0,1/(J+\sum_{j=0}^J \card{T_j})}$, and apply \Cref{lem:covering} to each $A_j$ for $1 \leq j \leq J$, and \Cref{lem:reluconc} to each $\param \in T$.
  By a union bound, with probability at least $1 - (J+\sum_{j=0}^J \card{T_j})\delta_0$, we have for all $\param \in \sphere$,
  \begin{align*}
    \E\sbr*{ \bfdelta_{\param} }
    & \leq \frac1n \sum_{i=1}^n \bfdelta_{\param}^i
    + 2C' \Delta_{r_0}(\param) \sqrt{\del*{ \frac1{\invtemp} + \Delta_{r_0}(\param) } \frac{\log(1/\delta_0)}{n}} + 2C' \Delta_{r_0}(\param) \frac{\log(1/\delta_0)}{n}
    \\
    & \qquad + \varepsilon_0 \Delta_{r_0}(\param)^{3/2} \del*{ \sqrt{\frac{2\ln(1/\delta_0)}{n}} + 2\sqrt{d} }
  \end{align*}
  where $\Delta_{r_0}(\param) := \max\set{ r_0, \norm{\param - \truth} }$, $C'>0$ is some absolute constant, and $\bfdelta_{\param}^i$ are as in \Cref{lem:reluconc}.
  We condition on the event that these inequalities hold for all $\param \in \sphere$.
  Since $\reluest$ minimizes $\frac1n \sum_{i=1}^n \sbr{-\bfy_i\bfx_i^\T\param}_+$, it follows that $\frac1n \sum_{i=1}^n \bfdelta_{\reluest}^i \leq 0$.
  Therefore, together with \Cref{lem:paramdiff} to lower-bound $\E\sbr*{ \bfdelta_{\param} }$ for $\param = \reluest$, we have
  \begin{multline*}
    \frac18\sqrt{\frac2\pi} \del*{1 - \frac1{\invtemp^2}} \norm{ \reluest - \truth }^2
    \leq
    2C' \Delta_{r_0}(\reluest) \sqrt{\frac{\log(1/\delta_0)}{\invtemp n}}
    + 2C' \Delta_{r_0}(\reluest)^{3/2} \sqrt{\frac{\log(1/\delta_0)}{n}}
    \\
    + 2C' \Delta_{r_0}(\reluest) \frac{\log(1/\delta_0)}{n}
    + \varepsilon_0 \Delta_{r_0}(\reluest)^{3/2} \del*{ \sqrt{\frac{2\ln(1/\delta_0)}{n}} + 2\sqrt{d} } .
  \end{multline*}
  Let $r_0 = \epsilon$, and suppose that $\norm{\reluest - \truth} > \epsilon$, in which case we have $\Delta_{r_0}(\reluest) = \norm{\reluest - \truth}$.
  Then, the display above has $\norm{\reluest - \truth}$ on both sides of the inequality, and it simplifies to an inequality in $x = \norm{\reluest - \truth}^{1/2}$ of the form $x - b\sqrt{x} - c \leq 0$ for some $b, c \geq 0$, which in turn implies $x \leq 1.5(b^2 + c)$.
  Therefore
  \begin{equation*}
    \norm{\reluest - \truth}
    \leq 
    C'' \del*{
      \varepsilon_0^2 \del*{ \frac{\log(1/\delta_0)}{n} + d }
      + \sqrt{\frac{\log(1/\delta_0)}{\invtemp n}}
      + \frac{\log(1/\delta_0)}{n}
    }
  \end{equation*}
  for some absolute constant $C''>0$.
  (Recall that we have assumed $\invtemp \geq 1 + c$ for some absolute constant $c>0$, and hence $1 - 1/\invtemp^2 \geq c(2+c)/(1+c)^2$.)
  Plug-in $\varepsilon_0 = \sqrt{\epsilon/(4C''d)}$, $\delta_0 = \delta/(J + \sum_{j=0}^J \card{T_j})$, and
  \begin{equation*}
    n \geq \frac{16\log(1/\delta_0)}{\invtemp \epsilon^2} + \frac{4C'' \log(1/\delta_0)}{\epsilon}
  \end{equation*}
  to conclude that $\norm{\reluest - \truth} \leq \epsilon$.
  The lower bound $n$ comes from the assumption in the theorem statement, and the fact that $J = O(\log(1/\epsilon))$ and $\sum_{j=0}^J \card{T_j} = O(d/\epsilon)^{O(d)}$.
\end{proof}

\subsection{Low temperatures}
\label{sec:lowtempupper}

The objective function minimized by $\reluest$ in \Cref{eq:reluest} uses the magnitude of the inner product $\bfx_i^\T\param$ whenever its sign differs from that of $\bfy_i$.
At sufficiently low temperatures ($\invtemp \gtrsim 1/\epsilon$), this magnitude information can be safely ignored.
Specifically, we show that the empirical (zero-one loss) risk minimizer
\begin{equation*}
  \erm((\bfx_i,\bfy_i)_{i=1}^n)
  \in \argmin_{\param \in \sphere}
  \frac1n \sum_{i=1}^n \ind{\bfy_i\bfx_i^\T\param \leq 0}
\end{equation*}
achieves near-optimal sample complexity in this regime.

\begin{theorem}
  \label[theorem]{thm:lowtempupper}
  Fix any $\epsilon, \delta \in \intoo{0,1}$, and assume
  \begin{equation*}
    \invtemp \geq \frac{4\sqrt{2\pi}}{\epsilon}
  \end{equation*}
  and
  \begin{equation*}
    n \geq \frac{C(d\log(1/\epsilon) + \log(1/\delta))}{\epsilon}
  \end{equation*}
  where $C>0$ is an absolute constant.
  Then with probability at least $1-\delta$, we have
  \begin{equation*}
    \norm{\erm((\bfx_i,\bfy_i)_{i=1}^n) - \truth} \leq \epsilon .
  \end{equation*}
\end{theorem}

Note that the sample size requirement given in \Cref{thm:lowtempupper} removes a $\log(d)$ factor from that in \Cref{thm:moderatetempupper} in the low temperature regime.

The proof of \Cref{thm:lowtempupper} is largely based on the following standard performance guarantee for empirical risk minimization, combined with the fact that the VC dimension of the class of homogeneous linear classifiers is $d$.\footnote{%
  Note that \Cref{lem:erm} is true under every distribution for $(\bfx,\bfy)$; it does not rely on specific properties of our model.
  In the special case of where $\bfx$ is the $d$-dimensional standard normal (or any other spherically symmetric distribution) and $\invtemp = \infty$ (which implies $\opt = 0$), the $\log(1/\varepsilon)$ can be removed~\citep{long1995sample}.
  So, in this case, the sample size requirement is $C(d + \log(1/\delta))/\varepsilon$.%
}

\begin{lemma}[\citealp{vapnik1971uniform}]
  \label[lemma]{lem:erm}
  There is a universal constant $C>0$ such that the following holds.
  Let $\opt = \min_{\param \in \sphere} \err_{\truth}(\param)$.
  For any $\varepsilon \in \intoo{0,1}$ and $\delta \in \intoo{0,1}$, if
  \begin{equation*}
    n \geq C \del*{ \frac{d\log(1/\varepsilon) + \log(1/\delta)}{\varepsilon^2} \opt + \frac{d\log(1/\varepsilon) + \log(1/\delta)}{\varepsilon} } ,
  \end{equation*}
  then the empirical risk minimizer $\erm$ satisfies
  \begin{equation*}
    \err_{\truth}(\erm) - \opt \leq \varepsilon
  \end{equation*}
  with probability at least $1-\delta$ over the realization of the data.
\end{lemma}

The following \namecref{lem:excesserr} (proved in \Cref{sec:excesserrproof}) relates $\norm{\erm - \truth}$ to $\err_{\truth}(\erm) - \err_{\truth}(\truth)$.
\begin{lemma}
  \label[lemma]{lem:excesserr}
  For any $\param \in \sphere$,
  \begin{equation*}
    \norm{\param - \truth}
    \leq
    \pi \del*{ \err_{\truth}(\param) - \err_{\truth}(\truth) } + \frac{2\sqrt{2\pi}}{\invtemp}
    .
  \end{equation*}
\end{lemma}

\begin{proof}[Proof of \Cref{thm:lowtempupper}]
  Note that
  \begin{equation*}
    \min_{\param \in \sphere} \err_{\truth}(\param)
    = \err_{\truth}(\truth)
    .
  \end{equation*}
  Therefore, we can combine \Cref{lem:erm} (with $\varepsilon = \epsilon/(2\pi)$) and \Cref{lem:excesserr} with the bound on $\opt = \err_{\truth}(\truth)$ from \Cref{lem:opt} to obtain the desired bound on $\norm{\erm - \truth}$.
\end{proof}

\subsection{Adaptivity}

If the inverse temperature $\invtemp$ is unknown, then we need to determine which of the aforementioned estimators to use in a data-driven fashion.
Notice that the estimators $\linest$, $\reluest$, and $\erm$ may be computed without explicit knowledge of $\invtemp$.
Therefore, it suffices to (coarsely) distinguish between the high ($\invtemp \lesssim 1$) and moderate-or-low ($\invtemp \gtrsim 1$) temperature regimes.
This can be done using an estimate of $\err_{\truth}(\truth)$ (e.g., training error rate of $\erm$) and reasoning about its relationship to $\invtemp$.

\section{Discussion}

Our characterization of the sample complexity of estimation in logistic regression delineates the high, moderate, and low temperature regimes.
However, we are only aware of computationally efficient estimators that achieve the (near) optimal sample complexities at high and zero temperatures: e.g., the linear estimator of \citet{plan2012robust} for $\invtemp \lesssim 1$, and the estimator based on solving a linear feasibility program (or the algorithm of~\citet{balcan2013active}) for $\invtemp = \infty$. Note that, although the ReLU loss is convex, we need to minimize it over the sphere.
It would be interesting to determine if the MLE itself (i.e., minimizing the logistic loss), or its efficient approximations, can shown to achieve optimal sample complexity. %

\section*{Acknowledgements}

We are most grateful to Akshay Krishnamurthy for helpful discussions about non-asymptotic analysis of MLE, and to the anonymous COLT reviewers for their constructive feedback and suggestions.
Part of this work was done while DH was visiting the Halıcıoğlu Data Science Institute at UC San Diego in Spring 2023.
We acknowledge the support of the National Science Foundation under grants IIS 2040971 and CCF 2217058/2133484.

%% file: paper-appendix.tex
\section{Bounds on Gaussian integrals}
\label{sec:integrals}

\begin{lemma}[\citealp{feller1968introduction}, page 175]
  \label[lemma]{lem:normaltail}
  Let $\bfz \sim \Normal(0,1)$.
  For any $t>0$,
  \begin{equation*}
    \del*{ 1 - \frac1{t^2} } \frac1{t\sqrt{2\pi}} \exp(-t^2/2)
    \leq
    \Pr(\bfz \geq t)
    \leq \frac1{t\sqrt{2\pi}} \exp(-t^2/2)
    .
  \end{equation*}
\end{lemma}

Recall that $g(\eta) = \ln(1+\exp(\eta))$.

\begin{lemma}
  \label[lemma]{lem:g''}
  Let $\bfz \sim \Normal(0,1)$.
  For any $\invtemp>0$,
  \begin{equation}
    \label{eq:g''1}
    \frac14 \E\sbr{ \exp\del{-\invtemp\abs{\bfz}} }
    \leq \E\sbr{ g''(\invtemp \bfz) }
    \leq \min\set*{
      \frac12
      , \
      \E\sbr{ \exp\del{-\invtemp\abs{\bfz}} }
    }
  \end{equation}
  and
  \begin{equation}
    \label{eq:g''2}
    \frac1{\invtemp} \del*{ 1 - \frac1{\invtemp^2} } \sqrt{\frac2\pi}
    \leq \E\sbr{ \exp\del{-\invtemp\abs{\bfz}} }
    \leq \frac1\invtemp \sqrt{\frac2\pi}
    .
  \end{equation}
\end{lemma}
\begin{proof}
  First, observe that
  \begin{equation*}
    g''(\eta)
    = \frac1{(1+\exp(\eta))(1+\exp(-\eta))}
    = \frac1{(1+\exp(\abs{\eta}))(1+\exp(-\abs{\eta}))}
    \leq \min\set*{
      \frac12 , \
      \exp(-\abs{\eta})
    } ,
  \end{equation*}
  while
  \begin{equation*}
    \frac{\exp(\abs{\eta})}{(1+\exp(\eta))(1+\exp(-\eta))}
    = \del*{ \frac1{1+\exp(-\abs{\eta})} }^2
    \geq \frac14 .
  \end{equation*}
  Plugging in $\eta = \invtemp\bfz$ and taking expectations gives the upper- and lower-bounds on $\E[g''(\invtemp\bfz)]$.
  For the upper- and lower-bounds on $\E[\exp(-\invtemp\abs{\bfz})]$, we have
  \begin{align*}
    \E\sbr{ \exp\del{-\invtemp\abs{\bfz}} }
    & = 2\int_0^\infty \frac1{\sqrt{2\pi}} \exp\del{ -\invtemp z - z^2/2} \dif z
    \\
    & = 2\exp\del{\invtemp^2/2}
    \int_0^\infty \frac1{\sqrt{2\pi}} \exp\del{ -(z+\invtemp)^2/2} \dif z
    \\
    & = 2\exp\del{\invtemp^2/2}
    \Pr(\bfz \geq \invtemp) ,
  \end{align*}
  and therefore the conclusion follows by applying \Cref{lem:normaltail}.
\end{proof}

\begin{lemma}
  \label[lemma]{lem:g'abs}
  Let $\bfz \sim \Normal(0,1)$.
  For any $\invtemp>0$,
  \begin{equation*}
    \E\sbr{ g'(-\invtemp\abs{\bfz}) }
    \leq
    \min\set*{
      \frac12
      , \
      \frac12 - \frac{\invtemp}4\sqrt{\frac2\pi} \del*{ 1 - \frac{\invtemp^2}6 }
      , \
      \frac1\invtemp \sqrt{\frac2\pi}
    }
    .
  \end{equation*}
\end{lemma}
\begin{proof}
  First, observe that $g'(\eta) \leq 1/2$ for all $\eta \leq 0$.
  Furthermore, by Taylor's theorem, for any $\eta\in\R$, there exists $h \in \R$ (between $0$ and $\eta$) such that
  \begin{align*}
    g'(\eta)
    & = g'(0) + g''(\eta) \eta + \frac12 g'''(\eta) \eta^2 + \frac16 g''''(h) \eta^3 \\
    & \leq g'(0) + g''(0) \eta + \frac12 g'''(0) \eta^2 + \frac1{48} \abs{\eta}^3 \\
    & = \frac12 + \frac14 \eta + \frac1{48} \abs{\eta}^3
    ,
  \end{align*}
  where the inequality uses the fact that $\abs{g''''(h)} \leq 1/8$ for all $h \in \R$.
  Therefore
  \begin{align*}
    \E\sbr{ g'(-\invtemp\abs{\bfz}) }
    & \leq \frac12 - \frac{\invtemp}4 \E\sbr{ \abs{\bfz} } + \frac{\invtemp^3}{48} \E\sbr{ \abs{\bfz}^3 } \\
    & = \frac12 - \frac{\invtemp}{4} \sqrt{\frac2\pi} + \frac{\invtemp^3}{24} \sqrt{\frac2\pi} \\
    & = \frac12 - \frac{\invtemp}{4} \sqrt{\frac2\pi} \del*{ 1 - \frac{\invtemp^2}{6} } .
  \end{align*}
  Finally, we also have
  \begin{equation*}
    \E\sbr{ g'(-\invtemp\abs{\bfz}) }
    = \E\sbr*{ \frac1{1+\exp\del{\invtemp\abs{\bfz}}} }
    \leq \E\sbr{ \exp\del{-\invtemp\abs{\bfz}} }
    \leq \frac1\invtemp \sqrt{\frac2\pi}
    ,
  \end{equation*}
  where the final inequality follows by \Cref{lem:g''}.
\end{proof}

\begin{lemma}
  \label[lemma]{lem:expqmoment}
  Let $\bfz \sim \Normal(0,1)$.
  For any $\invtemp>0$ and any non-negative integer $q$,
  \begin{equation*}
    \E\sbr{ \exp\del{-\invtemp\abs{\bfz}} \abs{\bfz}^q }
    \leq \frac{q!}{\invtemp^q} .
  \end{equation*}
\end{lemma}
We remark that the bound in \Cref{lem:expqmoment} can be improved to $O(q!/\invtemp^{q+1})$, which is tight up to constants.
This is because $\E\sbr{ \exp(-\invtemp\abs{\bfz}) \abs{\bfz}^q } = \sqrt{2/\pi} \mu_q m_{\invtemp}$, where $\mu_q$ is the $q$-th uncentered moment of the $\intco{0,\infty}$-truncated $\Normal(-\invtemp,1)$ distribution, and $m_{\invtemp} = (1-\Phi(\invtemp))/\phi(\invtemp)$ is the Mills ratio for $\Normal(0,1)$ at $\invtemp$.
However, we do not need this improved bound in the present work.
\begin{proof}[Proof of \Cref{lem:expqmoment}]
  The function $f \colon \R_+ \to \R$ defined by $f(x) = - \invtemp x + q \log(x)$ is maximized at $x = q/\invtemp$.
  Therefore
  \begin{align*}
    \E\sbr{ \exp\del{-\invtemp\abs{\bfz}} \abs{\bfz}^q }
    & = \E\sbr{ \exp\del{-\invtemp\abs{\bfz} + q \ln \abs{\bfz} } } \\
    & \leq \exp\del{-\invtemp (q/\invtemp) + q \ln(q/\invtemp) } \\
    & = \exp\del{-q + q \ln(q) - q\ln(\invtemp) } \\
    & \leq \exp\del{\ln(q!) - q\ln(\invtemp)} \\
    & = \frac{q!}{\invtemp^q} .
    \qedhere
  \end{align*}
\end{proof}

Recall that the Bernoulli distributions form an exponential family $\set{ p_\eta : \eta \in \R }$, where
\begin{equation*}
  p_\eta(y) = \exp\del*{ \eta \ind{y=1} - g(\eta) }
\end{equation*}
and $g(\eta) = \ln(1+\exp(\eta))$ is the log partition function for $p_\eta$.
The mean parameter---i.e., the mean of $\ind{y=1}$ under $p_\eta$---is given by $g'(\eta)$.
In the proof of the following \namecref{lem:klbound}, we use the fact that the KL divergence $\KL(p_\eta \| p_{\eta'})$ can be expressed as a Bregman divergence associated with $g$:
\begin{equation*}
  \KL(p_\eta \| p_{\eta'}) = g(\eta') - \del*{ g(\eta) + g'(\eta) (\eta' - \eta) } .
\end{equation*}

\begin{lemma}
  \label[lemma]{lem:klbound}
  Let $\bfz$ and $\bfz'$ be $\Normal(0,1)$ random variables with correlation $\rho$.
  Then
  \begin{equation*}
    \E\sbr*{
      \KL(\Bernoulli(g'(\invtemp\bfz)) \| \Bernoulli(g'(\invtemp\bfz')))
    }
    \leq 
    \frac{\invtemp}{2} (1-\rho) \min\set*{ \invtemp, 2\sqrt{2/\pi} }
  \end{equation*}
  and
  \begin{equation*}
    \E\sbr*{
      \KL(\Bernoulli(g'(\invtemp\bfz)) \| \Bernoulli(g'(\invtemp\bfz')))
    }
    \geq 
    \frac{\invtemp}{4} (1-\rho) \del*{ 1 - \frac1{\invtemp^2} } \sqrt{\frac2\pi}
    .
  \end{equation*}
\end{lemma}
\begin{proof}
  For any $\eta, \eta' \in \R$, we have
  \begin{equation*}
    \KL(\Bernoulli(g'(\eta)) \| \Bernoulli(g'(\eta')))
    = g(\eta') - g(\eta) - g'(\eta) \del{ \eta' - \eta }
    .
  \end{equation*}
  Hence,
  \begin{align*}
    \E\sbr{
      \KL(\Bernoulli(g'(\invtemp\bfz)) \| \Bernoulli(g'(\invtemp\bfz')))
    }
    & = \E\sbr{
      g(\invtemp \bfz')
      - g(\invtemp \bfz)
      - g'(\invtemp \bfz) \invtemp (\bfz' - \bfz)
    }
    \\
    & = \E\sbr{ g'(\invtemp \bfz) \invtemp \del{\bfz - \bfz'} }
    .
  \end{align*}
  Since $\bfz$ and $\bfz'$ have correlation $\rho$, we may write
  \begin{equation*}
    \bfz' = \rho\bfz + \sqrt{1-\rho^2} \bfz_\perp ,
  \end{equation*}
  where $\bfz_\perp \sim \Normal(0,1)$ is independent of $\bfz$.
  Thus
  \begin{align*}
    \E\sbr{ g'(\invtemp \bfz) \invtemp \del{\bfz - \bfz'} }
    & = \E\sbr*{ g'(\invtemp \bfz) \invtemp \del*{ \del{1-\rho} \bfz - \sqrt{1-\rho^2}\bfz_\perp } }
    \\
    & = \invtemp \del{1-\rho} \E\sbr{ g'(\invtemp \bfz) \bfz }
    \\
    & = \invtemp^2 \del{1-\rho}
    \E\sbr{ g''(\invtemp \bfz) }
    ,
  \end{align*}
  where the final step follows by Stein's identity.
  Now apply \Cref{lem:g''} to obtain the conclusion.
\end{proof}

\section{Proof of Lemma~\ref{lem:relumoments}}
\label{sec:relumomentsproof}

The following \namecref{lem:kvdgmoments} is implicit in the proofs of Lemma 5.2.1 and Lemma A.2.1 of~\citet{kuchelmeister2023finite}.

\begin{lemma}[\citealp{kuchelmeister2023finite}]
  \label[lemma]{lem:kvdgmoments}
  Fix $\truth, \param \in \sphere$ and $p \colon \R \to [0,1]$ satisfying $p(-t) = 1-p(t)$ for all $t \in \R$.
  Let $\bfx$ be a standard normal random vector in $\R^d$; let the conditional distribution of $\bfy$ given $\bfx$ be $\Bernoulli(p(\bfx^\T\truth))$; and define
  \begin{equation*}
    \bfdelta_{\param} := \sbr{ -\bfy \bfx^\T \param }_+ - \sbr{ -\bfy \bfx^\T \truth }_+ .
  \end{equation*}
  For any integer $q \geq 2$,
  \begin{multline*}
    \E \sbr*{ \abs{\bfdelta_{\param}}^q }
    \leq
    2^{q-1}
    \frac1{\pi\sqrt2} \Gamma\del*{ \frac{q+1}{2} } \del*{ \frac{\pi}{\sqrt2} \norm{\param - \truth} }^{q+1}
    + 2^{q-2} \norm{\param - \truth}^{2q} \E\sbr*{ p(-\abs{\bfz}) \abs{\bfz}^q }
    \\
    + 2^{2(q-1)} \norm{\param - \truth}^q \del*{ 1 - \frac14 \norm{\param - \truth}^2 }^{q/2} \frac{2^{q/2}}{\sqrt{\pi}} \Gamma\del*{ \frac{q+1}{2} } \E\sbr*{p(-\abs{\bfz})}
    .
  \end{multline*}
\end{lemma}
\begin{proof}
  Using the identity $\sbr{x}_+ = (\abs{x}+x)/2$, we have for any $\param \in \sphere$,
  \begin{align*}
    \bfdelta_{\param}
    = \sbr{ -\bfy \bfx^\T \param }_+ - \sbr{ -\bfy \bfx^\T \truth }_+
    & = 
    \frac{\abs{\bfx^\T\param} - \bfy\bfx^\T \param}{2}
    - \frac{\sign(\bfx^\T\truth)\bfx^\T\truth - \bfy\bfx^\T \truth}{2} \\
    & = 
    \frac{(\sign(\bfx^\T\param) - \sign(\bfx^\T\truth))\bfx^\T \param}{2}
    - \frac{(\bfy-\sign(\bfx^\T\truth))\bfx^\T(\param-\truth)}{2} .
  \end{align*}
  Therefore, for any $q \geq 2$, Jensen's inequality implies
  \begin{align*}
    \abs{\bfdelta_{\param}}^q
    & \leq 2^{q-1} \del*{
      \abs*{ \frac{(\sign(\bfx^\T\param) - \sign(\bfx^\T\truth))\bfx^\T \param}{2} }^q
      +
      \abs*{ \frac{(\bfy-\sign(\bfx^\T\truth))\bfx^\T(\param-\truth)}{2} }^q
    }
    \\
    & = 2^{q-1} \ind{\sign(\bfx^\T\param) \neq \sign(\bfx^\T\truth)} \abs{\bfx^\T \param}^q
    + 2^{q-1} \ind{\bfy \neq \sign(\bfx^\T\truth)} \abs{\bfx^\T(\param-\truth)}^q
    .
  \end{align*}
  \citet[Corollary A.2.1]{kuchelmeister2023finite} showed that
  \begin{align*}
    \E\sbr*{
      \ind{\sign(\bfx^\T\param) \neq \sign(\bfx^\T\truth)} \abs{\bfx^\T \param}^q
    }
    & \leq \frac1{\pi\sqrt2} \frac{\Gamma\del*{ \frac{q}{2} + 1 }}{q+1} \del*{ \frac{\pi}{\sqrt2} \norm{\param - \truth} }^{q+1} \\
    & \leq \frac1{\pi\sqrt2} \Gamma\del*{ \frac{q+1}{2} } \del*{ \frac{\pi}{\sqrt2} \norm{\param - \truth} }^{q+1}
  \end{align*}
  where the latter inequality uses $\Gamma(q/2+1)/(q+1) \leq \Gamma((q+1)/2)$.

  Using the conditional distribution of $\bfy$ given $\bfx$,
  \begin{align*}
    \E\sbr*{ \ind{\bfy \neq \sign(\bfx^\T\truth)} \abs{\bfx^\T(\param-\truth)}^q }
    & = \E\sbr*{ p(-\abs{\bfx^\T\truth}) \abs{\bfx^\T(\param-\truth)}^q } \\
    & = \E\sbr*{ p(-\abs{\bfz}) \abs{(\rho-1)\bfz + \sqrt{1-\rho^2}\bfz_\perp}^q }
  \end{align*}
  where $\rho = \param^\T\truth$, and $\bfz$ and $\bfz_\perp$ are independent standard normal random variables.
  By Jensen's inequality,
  \begin{equation*}
    \abs{(\rho-1)\bfz + \sqrt{1-\rho^2}\bfz_\perp}^q
    \leq 2^{q-1} \abs{\rho-1}^q \abs{\bfz}^q + 2^{q-1} (1-\rho^2)^{q/2} \abs{\bfz_\perp}^q
    .
  \end{equation*}
  Moreover, we have
  \begin{equation*}
    1 - \rho = \frac12 \norm{\param - \truth}^2
    \quad \text{and} \quad
    1 - \rho^2 = \norm{\param - \truth}^2 \del*{ 1 - \frac14 \norm{\param - \truth}^2 }
    .
  \end{equation*}
  Therefore, using independence of $\bfz$ and $\bfz_\perp$,
  \begin{align*}
    \lefteqn{
      \E\sbr*{ \ind{\bfy \neq \sign(\bfx^\T\truth)} \abs{\bfx^\T(\param-\truth)}^q }
    } \\
    & \leq 2^{q-1} \abs{\rho-1}^q \E\sbr*{ p(-\abs{\bfz}) \abs{\bfz}^q }
    + 2^{q-1} (1-\rho^2)^{q/2} \E\sbr*{p(-\abs{\bfz})} \E\sbr*{ \abs{\bfz_\perp}^q }
    \\
    & = \frac12 \norm{\param - \truth}^{2q} \E\sbr*{ p(-\abs{\bfz}) \abs{\bfz}^q }
    + 2^{q-1} \norm{\param - \truth}^q \del*{ 1 - \frac14 \norm{\param - \truth}^2 }^{q/2} \frac{2^{q/2}}{\sqrt{\pi}} \Gamma\del*{ \frac{q+1}{2} } \E\sbr*{p(-\abs{\bfz})}
    .
    \qedhere
  \end{align*}
\end{proof}

\begin{proof}[Proof of \Cref{lem:relumoments}]
  We use \Cref{lem:kvdgmoments} with $p(t) = g'(\invtemp t)$.
  Therefore we need to bound
  $\E\sbr{ g'(-\invtemp\abs{\bfz}) }$
  and
  $\E\sbr{ g'(-\invtemp\abs{\bfz}) \abs{\bfz}^q }$ for all integers $q\geq2$.
  Since
  $g'(-\invtemp\abs{z}) \leq \exp(-\invtemp\abs{z})$ for all $z \in \R$, we can use
  \Cref{lem:g''} for the former and
  \Cref{lem:expqmoment} for the latter.
  We obtain
  \begin{align*}
    \E \sbr*{ \abs{\bfdelta_{\param}}^q }
    & \leq
    2^{q-1} \cdot \frac1{\pi\sqrt2} \Gamma\del*{ \frac{q+1}{2} } \del*{ \frac{\pi}{\sqrt2} \norm{\param - \truth} }^{q+1}
    \\
    & \qquad
    +
    2^{q-2} \cdot \norm{\param - \truth}^{2q} \cdot \frac{q!}{\invtemp^q}
    \\
    & \qquad
    + 2^{2(q-1)} \cdot \norm{\param - \truth}^q \del*{ 1 - \frac14 \norm{\param - \truth}^2 }^{q/2} \frac{2^{q/2}}{\sqrt{\pi}} \Gamma\del*{ \frac{q+1}{2} } \cdot \frac1{\invtemp} \sqrt{\frac2\pi}
    \\
    & \leq
    2^{q-1} \cdot \frac1{\pi\sqrt2} \Gamma\del*{ \frac{q+1}{2} } \del*{ \frac{\pi}{\sqrt2} \norm{\param - \truth} }^{q+1}
    +
    2^{2q-2} \cdot \norm{\param - \truth}^q \cdot \frac{q!}{\invtemp}
    \\
    & \qquad
    + 2^{2(q-1)} \cdot \norm{\param - \truth}^q \frac{2^{q/2}}{\sqrt{\pi}} \Gamma\del*{ \frac{q+1}{2} } \cdot \frac1{\invtemp} \sqrt{\frac2\pi}
    \\
    & \leq
    q! \cdot \del*{ C \norm{\param - \truth} }^{q-2}
    \cdot \frac{C}{2} \cdot \norm{\param - \truth}^2
    \cdot \del*{ \norm{\param - \truth} + \frac1{\invtemp} }
  \end{align*}
  for some absolute constant $C>0$, where the final inequality uses $\Gamma((q+1)/2) \leq q!$.
  The final right-hand side in the previous display is clearly bounded above by
  $q!v b^{q-2} / 2$ for the specified choices of $b$ and $v$.
\end{proof}

\section{Proof of \Cref{lem:excesserr}}
\label{sec:excesserrproof}

\begin{proof}[Proof of \Cref{lem:excesserr}]
  We have
  \begin{align*}
    \err_{\truth}(\param)
    & = \E\sbr*{ g'(-\invtemp \sign(\bfx^\T\param) \bfx^\T\truth) } \\
    & = \E\sbr*{ g'(-\invtemp \sign(\bfx^\T\param) \bfx^\T\truth) \ind{\param^\T\bfx\bfx^\T\truth \leq 0} } + \E\sbr*{ g'(-\invtemp \sign(\bfx^\T\param) \bfx^\T\truth) \ind{\param^\T\bfx\bfx^\T\truth > 0} } \\
    & = \E\sbr*{ (1-g'(-\invtemp \abs{\bfx^\T\truth})) \ind{\param^\T\bfx\bfx^\T\truth \leq 0} } + \E\sbr*{ g'(-\invtemp \abs{\bfx^\T\truth}) \ind{\param^\T\bfx\bfx^\T\truth > 0} } \\
    & = \Pr(\param^\T\bfx\bfx^\T\truth \leq 0) - 2\E\sbr*{ g'(-\invtemp \abs{\bfx^\T\truth}) \ind{\param^\T\bfx\bfx^\T\truth \leq 0} } + \E\sbr*{ g'(-\invtemp \abs{\bfx^\T\truth}) } \\
    & = \Pr(\param^\T\bfx\bfx^\T\truth \leq 0) - 2\E\sbr*{ g'(-\invtemp \abs{\bfx^\T\truth}) \ind{\param^\T\bfx\bfx^\T\truth \leq 0} } + \err_{\truth}(\truth) .
  \end{align*}
  Therefore
  \begin{align*}
    \Pr(\param^\T\bfx\bfx^\T\truth \leq 0)
    & = \err_{\truth}(\param) - \err_{\truth}(\truth) + 2\E\sbr*{ g'(-\invtemp \abs{\bfx^\T\truth}) \ind{\param^\T\bfx\bfx^\T\truth \leq 0} } \\
    & \leq \err_{\truth}(\param) - \err_{\truth}(\truth) + 2\E\sbr*{ g'(-\invtemp \abs{\bfx^\T\truth}) } .
  \end{align*}
  The claim now follows by \Cref{lem:g'abs} and the fact that
  \begin{equation*}
    \Pr(\param^\T\bfx\bfx^\T\truth \leq 0)
    = \frac{\arccos(\param^\T\truth)}{\pi}
    \geq \frac{\sqrt{2(1-\param^\T\truth)}}{\pi}
    = \frac{\norm{\param - \truth}}{\pi}
    .
    \qedhere
  \end{equation*}
\end{proof}